\documentclass{proc-l-hijacked}
\usepackage{amssymb, amsmath, amsthm, hyperref,}

\newtheorem{theorem}{Theorem}[section]
\newtheorem{lemma}[theorem]{Lemma}

\newtheorem{proposition}[theorem]{Proposition}
\newtheorem{corollary}[theorem]{Corollary}

\theoremstyle{definition}
\newtheorem{definition}[theorem]{Definition}
\newtheorem{example}[theorem]{Example}

\DeclareMathOperator{\soc}{soc}

\DeclareMathOperator{\Ima}{Im}
\DeclareMathOperator{\Rea}{Re}

\numberwithin{equation}{section}

\begin{document}

	\title[]{Truncation and Spectral Variation in Banach Algebras}
	\author{C. Tour\'{e},  F. Schulz \and R. Brits}
	\address{Department of Mathematics, University of Johannesburg, South Africa}
	\email{cheickkader89@hotmail.com, francoiss@uj.ac.za, rbrits@uj.ac.za}

\begin{abstract}
Let $a$ and $b$ be elements of a semisimple, complex and unital Banach algebra $A$. Using subharmonic methods, we show that if the spectral containment $\sigma(ax)\subseteq\sigma(bx)$ holds for all $x\in A$, then $ax$ belongs to the bicommutant of $bx$ for all $x\in A$. Given the aforementioned spectral containment, the strong commutation property then allows one to derive, for a variety of scenarios, a precise connection between $a$ and $b$.  The current paper gives another perspective on the implications of the above spectral containment which was also studied, not long ago, by J. Alaminos, M. Bre\v{s}ar \emph{et. al.}
\end{abstract}
	\parindent 0mm
	
	\maketitle	

\section{Introduction}

Problems related to spectral variation under the multiplicative and additive operations in Banach algebras have recently attracted attention of researchers working in the field of abstract spectral theory in Banach algebras. Specifically, the first contributions were made by Bre\v{s}ar and \v{S}penko \cite{specproperties}, and at around the same time, but independently by Braatvedt and Brits \cite{univari}, and then later by J. Alaminos \emph{et. al.} \cite{specproperties2}, and Brits and Schulz \cite{univarisoc}. The aim of this paper is to extend and elaborate on the results obtained in \cite{specproperties2} and \cite{univarisoc}; we shall employ techniques which are distinctly different from the methods used in \cite{specproperties2} and \cite{specproperties}.

Unless otherwise stated, $A$ will assumed to be a semisimple, complex, and unital Banach algebra with the unit denoted by $\mathbf 1$. The group of invertible elements, and the centre of $A$ are denoted respectively by $G(A)$ and $Z(A)$. We shall use $\sigma_A$ and $\rho_A$ to denote, respectively, the spectrum
$$\sigma_A(x):=\{\lambda\in\mathbb C:\lambda\mathbf 1-x\notin G(A)\},$$ and the spectral radius
$$\rho_A(x):=\sup\{|\lambda|:\lambda\in\sigma_A(x)\}$$ of an element $x\in A$ (and agree to omit the subscript if the underlying algebra is clear from the context). Denote further by $\sigma^\prime(x):=\sigma(x)\backslash\{0\}$ the \emph{non-zero spectrum} of $x\in A$.
If $X$ is a compact Hausdorff space, then $A=C(X)$ is the Banach algebra of continuous, complex functions on $X$ with the usual pointwise operations and the spectral radius as the norm. If $X$ is a complex Banach space then $A=\mathcal L(X)$ is the Banach algebra of bounded linear operators on $X$ to $X$ (also in the usual sense). The main question of this paper is, loosely stated, the following:\medskip

\emph{ Let $A$ be a semisimple, complex, and unital Banach algebra, and suppose that $a,b\in A$ satisfy
	\begin{equation}\label{contain}
	\sigma(ax)\subseteq\sigma(bx)\mbox{ for all }x\in A.
	\end{equation}
	What is the relationship between $a$ and $b$?}
\medskip

Observe, trivially, that
$$\eqref{contain}\Rightarrow\sigma^\prime(ax)\subseteq\sigma^\prime(bx)\mbox{ for all }x\in A.$$ 
Since the non-zero spectrum is cyclic (Jacobson's Lemma, \cite[Lemma 3.1.2]{aupetit1991primer}) it turns out to be  advantageous to assume, where applicable, the preceding implication of \eqref{contain} rather than \eqref{contain} itself. For easy reference we label
\begin{equation}\label{contain2}
\sigma^\prime(ax)\subseteq\sigma^\prime(bx)\mbox{ for all }x\in A,
\end{equation}
and then note that \eqref{contain2} is equivalent to the statement: $$\sigma^\prime(xa)\subseteq\sigma^\prime(xb)\mbox{ for all }x\in A.$$

Further, if \eqref{contain2} holds then we also have 
\begin{equation}\label{contain3}
\sigma^\prime((b-a)x)\subseteq\sigma^\prime(bx)\mbox{ for all }x\in A.
\end{equation}
To see this, if $\lambda\not=0$ and $\lambda\notin\sigma^\prime(bx)$, then 
$$1+bx(\lambda\mathbf1-bx)^{-1}=\lambda(\lambda\mathbf1-bx)^{-1}\in G(A),$$
from which the assumption \eqref{contain2} implies that $\mathbf1+ax(\lambda\mathbf1-bx)^{-1}\in G(A).$
Then $$\lambda\mathbf 1-(b-a)x=(\mathbf1+ax(\lambda\mathbf1-bx)^{-1})(\lambda\mathbf 1-bx)\in G(A).$$

We give a short list of some of the major known results which are related to \eqref{contain} and \eqref{contain2}:

\begin{itemize}
	\item[(a)]{ \cite[Theorem 3.7]{specproperties}: Let $A$ be a prime $C^\star$-algebra and let $a,b\in A$ be such that $\rho(ax)\leq\rho(bx)$ for all $x\in A$. Then there exists $\lambda\in\mathbb C$ such that $|\lambda|\leq1$ and $a=\lambda b$. }	
	\item[(b)]{\cite[Theorem 2.6]{univari}: If $A$ is an arbitrary semisimple, complex and unital Banach algebra, and $a,b\in A$, then $a=b$ if and only if $\sigma(ax)=\sigma(bx)$ for all $x\in A$ satisfying $\rho(x-\mathbf 1)<1$ (the bound on the spectral radius is sharp).  }
	\item[(c)]{ \cite[Theorem 2.3]{specproperties2}: If $A$ is a unital $C^\star$-algebra and $a,b\in A$, then $\sigma(ax)\subseteq\sigma(bx)\cup\{0\}$ for every $x\in A$ if and only if there exists a central projection $z\in A^{\prime\prime}$, the second dual of $A$, such that $a=zb$. }	
	\item[(d)]{ \cite[Theorem 3.6]{specproperties2}: If $A$ is a unital $C^\star$-algebra and $a,b\in A$, then $\rho(ax)\leq\rho(bx)$ for every $x\in A$ if and only if there exists a central projection $z\in A^{\prime\prime}$, the second dual of $A$, such that $a=zb$ and $\|z\|\leq1$. }	
	\item[(e)]{\cite[Theorem 3.9]{univarisoc}: If $A$ is a semisimple, complex and unital Banach algebra with non-zero socle, denoted $\soc(A)$, then $A$ is prime if and only if for $a,b\in A$ the following are equivalent:
		\begin{itemize}
			\item[(i)]{$\rho(ax)\leq\rho(bx)$ for all $x\in A$.}
			\item[(ii)]{$a=\lambda b$ for some $\lambda\in\mathbb C$ with $|\lambda|\leq1$. }
		\end{itemize} In particular, if $A=\mathcal L(X)$, then (i) and (ii) are equivalent.  }		
\end{itemize}

The following simple example serves as the impetus for this paper, and may perhaps indicate a general relationship between $a$ and $b$ when \eqref{contain} is satisfied:

\begin{example}\label{truncation}
	Let $A=C(X)$ where $X=[0,1]$. Define $b,a\in A$ by respectively 
	\begin{equation}
	b(t)=\left|t-1/2\right|,\ t\in X\mbox { and } a(t)=\left\{\begin{array}{cc} 0, & t\in[0,1/2)\\
	t-1/2, & t\in[1/2,1].\end{array}\right.
	\end{equation}
\end{example}

Then \eqref{contain} holds, and moreover, from the graphs of $a$ and $b$, it is easy that $a$ is a truncation of $b$. The obvious question is whether, for arbitrary Banach algebras, \eqref{contain} or \eqref{contain2}, implies that $a$ is, in some suitable sense, a ``truncation" of $b$? Example~\ref{truncation} suggests the following definition:

\begin{definition}[algebraic truncation]
	Let $A$ be a complex and unital Banach algebra, and let $a,b\in A$. Then $a$ is said to be an \emph{algebraic truncation} of $b$ if 
	$$a(b-a)=(b-a)a=0.$$
\end{definition}

\section{General results}
We start with a simple but interesting observation:
\begin{proposition}\label{reverse}
	If $a,b\in A$ satisfy $ax(b-a)=0$ for all $x\in A$, then  $$\sigma^\prime(ax)\subseteq\sigma^\prime(bx)\mbox{ for all }x\in A.$$
\end{proposition} 
\begin{proof}
	We shall first prove that $ax(b-a)=0$ for all $x\in A$ implies that
	$(b-a)xa=0$ for all $x\in A$. With the hypothesis, suppose that 
	$(b-a)x_0a\not=0$ for some $x_0\in A$. Since $A$ is semisimple we can find $y\in A$ such that $\sigma((b-a)x_0ay)\not=\{0\}$. But this gives a contradiction because  $((b-a)x_0ay)^2=0$. Towards the spectral containment: If we write $ax+(b-a)x=bx$ then  the preceding calculation implies that 
	$$\sigma^\prime(bx)=\sigma^\prime(ax)\cup \sigma^\prime\left((b-a)x\right),$$ which proves the claim.

\end{proof}		

In light of Proposition~\ref{reverse} the main question can therefore be phrased as whether the condition  $ax(b-a)=0$ for all $x\in A$ is the only possible instance which fulfils the spectral containment \eqref{contain2}. Notice further that the condition  $ax(b-a)=0$ for all $x\in A$ is equivalent to the condition $ax(b-a)x=0$ for all $x\in A$; in other words, to the condition that $ax$ is an algebraic truncation of $bx$ for all $x\in A$. 

Our first main result shows that \eqref{contain2} forces strong commutation properties. The proof is an application of Vesentini's Theorem \cite[Theorem 3.4.7]{aupetit1991primer}:

\begin{theorem}\label{commute}
	If  $\sigma^\prime(ax)\subseteq\sigma^\prime(bx)\mbox{ for all }x\in A$ then $a$ belongs to the bicommutant of $b$. Hence, for each $x\in A$, $ax$ belongs to the bicommutant of $bx$.
\end{theorem}

\begin{proof}
	Pick $\alpha\in\mathbb C$ arbitrary but fixed, and suppose $c\in A$ commutes with $b$.  By assumption, we have that
	$$\sigma^\prime(ae^{\alpha c}xe^{-\alpha c})\subseteq\sigma^\prime(be^{\alpha c}xe^{-\alpha c})\mbox{ for all }x\in A.$$
	Jacobson's Lemma, together with the fact that $b$ and $c$ commute, imply that  	
	$$\sigma^\prime(e^{-\alpha c}ae^{\alpha c}x)\subseteq\sigma^\prime(bx)\mbox{ for all }x\in A.\eqno(1)$$
	Now fix $x$ and take $\lambda\in\mathbb C$ with $\rho(bx)<|\lambda|$. Then
	$\mathbf1+bx(\lambda\mathbf 1-bx)^{-1}\in G(A)$ from which (1) implies that 
	$\mathbf1+e^{-\alpha c}ae^{\alpha c}x(\lambda-bx)^{-1}\in G(A)$ and so we have
	$\mathbf1+(\lambda-bx)^{-1}e^{-\alpha c}ae^{\alpha c}x\in G(A)$. Multiplication by $\lambda\mathbf1-bx$ on the left then implies that
	$\lambda\mathbf1-(bx-e^{-\alpha c}ae^{\alpha c}x)\in G(A)$. Since $\rho(bx)<|\lambda|$ we observe that (1) implies $\lambda\mathbf1+e^{-\alpha c}ae^{\alpha c}x\in G(A)$. Arguing as before, factorizing
	$$(\lambda\mathbf1+e^{-\alpha c}ae^{\alpha c}x)(\mathbf1-(\lambda\mathbf1+e^{-\alpha c}ae^{\alpha c}x)^{-1}bx),$$ followed by multiplication with $(\lambda\mathbf1+e^{-\alpha c}ae^{\alpha c}x)^{-1}$ on the left, we have that $\mathbf1-(\lambda\mathbf1+e^{-\alpha c}ae^{\alpha c}x)^{-1}bx\in G(A)$ whence $\mathbf1-(\lambda\mathbf1+e^{-\alpha c}ae^{\alpha c}x)^{-1}ax\in G(A)$.
	Then 
	\begin{align*}
	(\lambda\mathbf1+e^{-\alpha c}ae^{\alpha c}x)[\mathbf1-(\lambda\mathbf1+e^{-\alpha c}ae^{\alpha c}x)^{-1}ax]&=\lambda\mathbf1-[ax-e^{-\alpha c}ae^{\alpha c}x]\in G(A).
	\end{align*}
	From this it follows that, for each $\alpha\in\mathbb C$,
	$$\rho(ax-e^{-\alpha c}ae^{\alpha c}x)\leq\rho(bx).$$
	The subharmonic function 
	$$\alpha\mapsto\rho(ax-e^{-\alpha c}ae^{\alpha c}x)$$ is therefore bounded on $\mathbb C$, and, by Liouville's Theorem, it must be constant. In particular, with $\alpha=0$, we see that it vanishes everywhere on $\mathbb C$. Define $f:\mathbb C\rightarrow A$ by 
	\begin{displaymath}
	f(\alpha)=\left\{\begin{array}{cc} [ax-e^{-\alpha c}ae^{\alpha c}x]/\alpha  & \alpha\not=0\\
	(ca-ac)x & \alpha=0.\end{array}\right.
	\end{displaymath}
	Then $f$ is analytic on $\mathbb C$ and $\rho(f(\alpha))=0$ holds for all $\alpha\not=0$. But this means that $\rho(f(0))=0$. Since $x$ was arbitrary, and $A$ is semisimple, we have that $ca-ac=0$ as required.
\end{proof}	

\begin{corollary}\label{cyc}
	If  $\sigma^\prime(ax)\subseteq\sigma^\prime(bx)\mbox{ for all }x\in A$ then $axb=bxa$ for all $x\in A$.
\end{corollary}
\begin{proof}
	Theorem~\ref{commute} says that for each $x\in A$, $axbx=bxax$. Now, given $x\in A$, pick $\lambda\in\mathbb C$ such that $\lambda\mathbf 1-x$ is invertible. Then obviously $a(\lambda\mathbf 1-x)b=b(\lambda\mathbf 1-x)a$. So $axb=bxa$ follows from $ab=ba$. 	
\end{proof}	

To obtain one of our main results in this section, Theorem~\ref{group}, we shall need two lemmas. The first is somewhat folklore, but very well-known, and appears scattered throughout the literature on Banach algebras; the second lemma is, as far as the authors could establish, originally due to Ptak \cite{ptakderivations} and has since been ``rediscovered", and applied, in a number of papers related to Banach algebra theory. 

\begin{lemma}\label{folklore}
	If $A$ is a semisimple, complex and unital Banach algebra, and $p\in A$ is a projection, then $pAp$ is a semisimple Banach algebra with identity element $p$. Moreover
	$$\sigma^\prime_{pAp}(z)=\sigma^\prime_{A}(z)\mbox{ holds for each }z\in pAp.$$
\end{lemma}

\begin{lemma}[Ptak]\label{center}
	If $A$ is a semisimple, complex and unital Banach algebra, and $z\in A$ satisfies $\rho(zx)\leq \rho(x)$ for all $x \in A$, then $z \in Z(A)$. 
\end{lemma}

\begin{theorem}\label{group}
	Suppose, for some $a,b\in A$, that $\sigma^\prime(ax)\subseteq\sigma^\prime(bx)$ for all $x\in A$.
	\begin{itemize}
		\item[(a)]{ If $a$ is invertible then $a=b$.}	
		\item[(b)]{ If $b$ is a projection then $ab=ba=a$ and $a$ is a projection. In particular if $b=\mathbf 1$, then $a$ is a projection belonging to $Z(A)$.}	
		\item[(c)]{ If $b$ is invertible then $a$ is group invertible. In particular, there exists a projection $p\in Z(A)$ such that $a=bp$.}
		\item[(d)]{ If $a$ is a projection then $ab=ba=a$.}	
	\end{itemize}
\end{theorem}	
\begin{proof}
	(a) We shall first prove that if $a=\mathbf 1$ then $b=\mathbf 1$. 	Observe that, by Corollary~\ref{cyc}, $b\in Z(A)$. To obtain the preliminary result we consider two cases:
	
	(i){ $b\in G(A)$: Obviously $b^{-1}\in Z(A)$ and, moreover, we have that $\sigma(b^{-1})\subseteq \sigma(bb^{-1})=\{1\}$ implies that $\sigma(b^{-1})=\{1\}$. Then $\rho((b^{-1}-\mathbf1)x)\leq \rho(b^{-1}-\mathbf1)\rho(x)=0$ implies (by semisimplicity) that $b^{-1}-\mathbf 1 =0$. Thus $b=\mathbf 1$.}\smallskip

	(ii){ $b\notin G(A)$: Pick $\lambda\in\mathbb C$ such that $\lambda \mathbf{1}-bx,\,\lambda\mathbf{1} \in G(A)$. Then
		$$(\lambda \mathbf{1}-bx)(1+(\lambda\mathbf{1}-bx )^{-1}bx)=\lambda \mathbf{1}\in G(A)$$
		from which it follows that $ 1+( \lambda \mathbf{1}-bx )^{-1}bx \in G(A)$ and hence that $1+( \lambda \mathbf{1}-bx )^{-1}x\in G(A)$ (using \eqref{contain2} together with the fact that $(\lambda \mathbf{1}-bx )^{-1}$ commutes with $x$). Thus
		$$\lambda \mathbf{1}-(b-\mathbf1)x=(\lambda \mathbf{1}-bx )(1+(\lambda \mathbf{1}-bx )^{-1}x) \in G(A)$$  and so $\sigma{'}((b-\mathbf1)x) \subseteq \sigma{'}(bx) $. Since $b\in Z(A)$ we have that $bx$ is not invertible for any $x \in A$ and we infer that $\sigma((b-\mathbf1)x) \subseteq \sigma(bx) $. Taking $x=\mathbf 1$ we deduce $\sigma(b-\mathbf1) \subseteq \sigma(b)$. But since $0 \in \sigma(b)$ the implication of this  would (inductively) be that all negative integers belong to $\sigma(b)$ contradicting the compactness of the spectrum. Thus, if $a=\mathbf 1$ then $b=\mathbf 1$.} 
	To complete the proof of (a) notice that if $a\in G(A)$ then $\eqref{contain2}$ implies that $\sigma^\prime(x)\subseteq\sigma^\prime(ba^{-1}x)$ holds for all $x\in A$. So by the preceding paragraph $\mathbf 1=ba^{-1}$ and the result follows.
	
	(b) From the hypothesis we deduce that $\sigma^\prime(abbxb)\subseteq\sigma^\prime(bbxb)$, and hence that	$\sigma^\prime(abbxb)\subseteq\sigma^\prime(bxb)$ for all $x \in A$. Denote by $B$ the semisimple Banach algebra $bAb$. Using Theorem~\ref{commute} it follows that $ab=bab\in B$, and from Lemmas~\ref{folklore} and ~\ref{center} we deduce that $ab$ commutes with every $c\in B$. Since $b$ is a projection in $A$ we have that $\sigma_B(ab)\subseteq \{0,1\}.$ Therefore $\sigma_B(ab(ab-b))=\{0\}$ from which
	$$\rho_B(ab(ab-b)c)\leq \rho_B(ab(ab-b))\rho_B(c)=0\mbox{ for each }c\in B.$$	Since $B$ is semisimple we conclude that $ab(ab-b)=0$, and hence that $ab$ is a projection. But the hypothesis also implies that $ \sigma(a(b-\mathbf 1)x) \subseteq \sigma(b(b-\mathbf 1)x)=\{0\} $ whence $a(b-\mathbf 1)=0$ by semisimplicity. Consequently $ab=a$ is a projection.
	
	(c) If $b\in G(A)$ then $\eqref{contain2}$ implies that $\sigma^\prime(ab^{-1}x)\subseteq\sigma^\prime(x)$ holds for all $x\in A$. It follows from part (b) that $ab^{-1}=p$ for some projection $p$ in $Z(A)$.
	
	(d) Observe that $\sigma_{aAa}^{\prime}(a(axa))\subseteq \sigma_{aAa}^{\prime}(aba(axa))$ holds in the semisimple Banach algebra $aAa$ which has identity element $a$. It follows from part (a) that $a=aba$, and hence that $ab=ba=a$. 
	
\end{proof}

The following is immediate from Theorem~\ref{group}:

\begin{corollary}\label{abinvert}
	Suppose, for some $a,b\in A$, that $\sigma^\prime(ax)\subseteq\sigma^\prime(bx)$ for all $x\in A$. If either $a$ or $b$ is invertible, or a projection, then $a$ is an algebraic truncation of $b$.  
\end{corollary}	

Theorem~\ref{scalar} settles the case, with respect to \eqref{contain2}, when $a$ and $b$ are linearly dependent. As a corollary we can then deduce a precise algebraic characterization of \eqref{contain2} for some important classes of Banach algebras.

\begin{theorem}\label{scalar}
	If $a=\alpha b$ for some $\alpha\in\mathbb C$, and $\sigma^\prime(ax)\subseteq\sigma^\prime(bx)$ for all $x\in A$, then $\alpha=0$ or $\alpha=1$. 
\end{theorem}
\begin{proof}
	If $a$ or $b$ is invertible then $a=b$ or $a=0$, and the proof is complete; so we may assume $a,b\not\in G(A)$. If $\sigma(bx)=\{0\}$ for all $x\in G(A)$, then, by semisimplicity, $a=b=0$. We can therefore assume the existence of $x^\prime\in G(A)$ such that $\sigma(bx^\prime)\not=\{0\}$. If we can establish $ax^\prime=0$ or $ax^\prime=bx^\prime$, then $a=0$ or $a=b$. So we can assume, without loss of generality, that $\sigma(b)\not=\{0\}$. To obtain a contradiction we shall assume then that $\alpha\not=0$, $\alpha\not=1$. The first step is to show that $\alpha\in(0,1)\subset\mathbb R$: Via the spectral radius we obtain $|\alpha|\leq1$. But \eqref{contain3} implies that $1-\alpha$ is also a number satisfying 
	$$\sigma^\prime((1-\alpha)bx)\subseteq\sigma^\prime(bx)\mbox{ for all }x\in A$$ whence $|1-\alpha|\leq1$.  
	Observe next that if $\alpha$ is a number satisfying $\sigma^\prime(\alpha bx)\subseteq\sigma^\prime(bx)\mbox{ for all }x\in A,$ then so is the number $\alpha^n$ for any $n\in\mathbb N$. Arguing as before we therefore have 
	$$\sigma^\prime((\alpha^n bx)\subseteq\sigma^\prime(bx)\mbox{ and }\sigma^\prime((1-\alpha^n)bx)\subseteq\sigma^\prime(bx)\mbox{ for all }x\in A, n\in\mathbb N.$$ But if $\alpha\not\in(0,1)$ then, for some $n\in\mathbb N$, $\alpha^n$ would not be in the ``feasible region" (by rotation) i.e. $\alpha^n\notin \{\lambda\in\mathbb C:|\lambda|\leq1\}\cap \{\lambda\in\mathbb C:|1-\lambda|\leq1\}$. So we conclude that $\alpha\in(0,1)$. At this stage we need to make two further observations: Firstly, if $\alpha\in (0,1)$ and $\sigma^\prime(\alpha bx)\subseteq\sigma^\prime(bx)$ holds for all $x\in A$, then,  given any $\epsilon>0$, we can (by the preceding argument) find $\beta\in(0,1)$ such that $|\beta|<\epsilon$ and  $\sigma^\prime(\beta bx)\subseteq\sigma^\prime(bx)$. Consequently we can also find $\gamma\in(0,1)$ such that $|\gamma-1|<\epsilon$ and $\sigma^\prime(\gamma bx)\subseteq\sigma^\prime(bx)$ for all $x\in A$. Secondly, if 
	$\sigma^\prime(\alpha bx)\subseteq\sigma^\prime(bx)$ holds for all $x\in A$, and $\xi\in\mathbb C$ is arbitrary then
	$\sigma^\prime(\alpha (\xi b)x)\subseteq\sigma^\prime((\xi b)x)$ holds for all $x\in A$. Thus, by the second observation, we can assume without loss of generality that $\sigma^\prime(b)$ contains a complex number on a horizontal line $y=\pm(2k+1)\pi$ (for some $k\in\mathbb N$) in the complex plane. Now, if we take 
	$$x_0=\sum_{j=0}^\infty\frac{b^j}{(j+1)!}$$ then 
	$$e^b-1=\sum_{j=1}^\infty \frac{b^j}{j!}=bx_0.$$ Since $\sigma^\prime(b)$ is bounded and contains at least one complex number on the horizontal line $y=\pm(2k+1)\pi$ for some $k\in\mathbb N$, it follows, by the Spectral Mapping Theorem, that there exists an open interval $(t_1,t_2)\subset\mathbb R$ with $t_1<t_2<0$ such that $t_1\in\sigma^\prime(bx_0)$ and $(t_1,t_2)\cap\sigma^\prime(bx_0)=\emptyset$. But now, by the first observation, with $|\gamma-1|$ sufficiently small, we obtain a contradiction with $\sigma^\prime(\gamma bx_0)\subseteq\sigma^\prime(bx_0)$. Thus either $\alpha=0$ or $\alpha=1$. 
	
\end{proof}

For the last paragraph of this section we recall that an algebra $A$ is called \emph{prime} if $$axb=0\mbox{ for all }x\in A\Rightarrow a=0\mbox{ or } b=0.$$ Further, a prime Banach algebra $A$ is called \emph{centrally closed} if the \emph{extended centroid} (see Sections 7.4--7.6 in \cite{bresar2014} for the definition and properties) of $A$ is equal to the complex field.  We can now establish:

\begin{corollary}\label{LX}
	Let $A$ be a centrally closed semisimple prime Banach algebra, and let $a,b\in A$. Then $\sigma^\prime(ax)\subseteq\sigma^\prime(bx)$ for all $x\in A$ if and only if either  $a=0$ or $a=b$.  		
\end{corollary}

\begin{proof}
	Notice that the condition $axb=bxa$ for all $x\in A$ (Corollary~\ref{cyc}) implies that $a$ and $b$ are linearly dependent over the extended centroid \cite[Lemma 7.41]{bresar2014}. The forward implication is then clear from Theorem~\ref{scalar}. The reverse implication is trivial.	
\end{proof}

Corollary~\ref{LX} covers some important classes of Banach algebras: Examples include prime $C^\star$-algebras (see for instance \cite[Proposition 2.2.10]{localmultipliers} as well as primitive Banach algebras (look at \cite[Corollary 4.1.2]{ringsgenid}); specifically,  Corollary~\ref{LX} holds for $A=\mathcal L(X)$. 
Finally we may observe that semisimple prime algebras can be characterized in terms of the spectral containment \eqref{contain2}: 

\begin{proposition}\label{prime}
	$A$ is prime if and only if for $a\not=0,\ a\not=b$
	\begin{equation}\label{primecond}
	\sigma^\prime(ax)\subseteq\sigma^\prime(bx)\ \forall\ x\in A\  \Rightarrow aA\cap (b-a)A\not=\{0\}.
	\end{equation}
\end{proposition}
\begin{proof}
	Suppose $A$ is prime, and there exist $a\not=0$, $b\not=a$ in $A$ satisfying 
	$\sigma^\prime(ax)\subseteq\sigma^\prime(bx)$ for all $x\in A$ with $aA\cap (b-a)A=\{0\}$. By Theorem~\ref{commute} we then have $(b-a)xa=ax(b-a)$ for all $x\in A$, from which it follows that $ax(b-a)=0$ for all $x\in A$. But, since $A$ is prime, this gives a contradiction. For the converse, suppose $A$ is not prime. Then, since $A$ is semisimple, we can find $a,c\in A$, both nonzero, and $a\not=c$, such that $axc=0$ for all $x\in A$. It then follows from Jacobson's Lemma, together with semisimplicity of $A$, that $cxa=0$ for all $x\in A$. Now set $b=c+a$. Then 
	$$\sigma^\prime(ax)\subseteq\sigma^\prime(cx)\cup\sigma^\prime(ax)=\sigma^\prime(cx+ax)=\sigma^\prime(bx)$$ holds for each $x\in A$. Suppose
	$aA\cap (b-a)A\not=\{0\}$. Then we find can $x,y\in A$ such that $ax=cy\not=0$. Orthogonality implies that, for each $z\in A$, we have $\sigma(axzax)=\{0\}$ from which Jacobson's Lemma, semisimplicity of $A$, and the Spectral Mapping Theorem yield the contradiction that $ax=0$.  So we conclude that if $A$ is not prime then \eqref{primecond} does not hold which completes the proof.
\end{proof}

Intuitively, the impression is that Proposition~\ref{prime} seems a bit artificial; the natural conjecture here should be that a semisimple Banach algebra $A$ is prime if and only if 
$$\sigma^\prime(ax)\subseteq\sigma^\prime(bx)\ \forall\ x\in A\  \Rightarrow a=0\mbox{ or }a=b.$$
However, a complete proof of the preceding statement eludes us at this stage, thus leaving the problem as a conjecture.

\section{$C^\star$-algebras}

As an application of the results in the preceding section we consider the case where $A$ is a $C^\star$-algebra. We should point out that Theorem~\ref{Cstar} can also very easily be obtained as a corollary of \cite[Theorem 2.3]{specproperties2}, and it is therefore not really new; the main difference here lies in the arguments leading to the respective results. We first establish the result for the commutative case:

\begin{theorem}\label{C(X)}
	Let $a,b\in A=C(X)$ where $X$ is any compact Hausdorff space. Then $\sigma^\prime(ax)\subseteq\sigma^\prime(bx)\mbox{ for all }x\in A$,  if and only if $a$ is an algebraic truncation of $b$. 
\end{theorem}

\begin{proof}
	The reverse implication follows directly from Proposition~\ref{reverse}. We prove the forward implication: If either $a$ or $b$ belongs to $G(A)$ then the result follows from Corollary~\ref{abinvert}. So we assume neither is invertible. For $x\in A$ denote by $x^\star \in A$ the function $x^\star(t)=\overline{x(t)}$. Using 
	\eqref{contain2} we obtain
	$$\sigma^\prime(aa^\star x)\subseteq\sigma^\prime(ab^\star x)\subseteq\sigma^\prime(bb^\star x)\mbox{ for all }x\in A.$$
	In particular $\sigma(ab^\star )\subset\mathbb R^+$ which means that, for each $t\in X$, $a(t)\overline{b(t)}\in\mathbb R^+$. This shows that $ab^\star =a^\star b$. For each $x\in A$ define $K_x=\{t\in X:x(t)=0\}$. We can write
	\begin{equation}
	x(t)=\left\{\begin{array}{cc} r_x(t)\left[\cos\theta_x(t)+i\sin\theta_x(t)\right], & x(t)\not=0\\
	0, & x(t)=0.\end{array}\right.
	\end{equation}
	where $r_x$ and $\theta_x$ are functions: 
	$$r_x:K_x\rightarrow \mathbb R^+\mbox { and }\theta_x:K_x\rightarrow (-\pi,\pi]$$ 
	Suppose now for some $t\in X$, $a(t)\not=0$ and $b(t)\not=0$. Then 
	$$a(t)b^\star (t)=r_a(t)r_b(t)[\left[\cos(\theta_a(t)-\theta_b(t))+i\sin(\theta_a(t)-\theta_b(t))\right]\in\mathbb R^+$$ implies that $\sin(\theta_a(t)-\theta_b(t))=0$, from which we deduce
	(i) $\theta_a(t)-\theta_b(t)=0$ or (ii) $\theta_a(t)-\theta_b(t)=\pm\pi$. If (ii) holds then we have a contradiction with $\sigma(ab^\star )\subset\mathbb R^+$ and so  $\theta_a(t)=\theta_b(t)$. We can hence conclude from this that for each $t\in X$ such that $a(t)\not=0$ and $b(t)\not=0$ there exists a corresponding positive real number, say $\alpha(t)$, such that $a(t)=\alpha(t)b(t)$. Formally we have the following:
	If $K=K_a\cup K_b$, and if $a$ and $b$ satisfy \eqref{contain2}, then there exists a continuous function $\alpha: X\backslash K\rightarrow (0,\infty)$ such that $$a(t)=\alpha(t)b(t)\mbox{ for all }t\in X\backslash K.$$
	We now show that the function $\alpha(t)\equiv1$ for all $t\in X\backslash K$: By $\eqref{contain2}$ we have that $\sigma^\prime\left(ab^\star +iab^\star bb^\star \right)\subseteq  \sigma^\prime(bb^\star +i(bb^\star )^2)$. Observe that the set on the right side is contained in the parabola $\Ima(z)=[\Rea(z)]^2$ in the first quadrant of the complex plane.
	Using the fact that $a(t)=\alpha(t)b(t)$ for all $t\in X\backslash K$ where $\alpha(t)\in(0,\infty)$ yields the containment
	$$\{\alpha(t)bb^\star (t)+i\alpha(t)(bb^\star (t))^2:t\in X\backslash K\}\subseteq \{bb^\star (t)+i(bb^\star (t))^2:t\in X\backslash K\}.$$  So if $t_0$ belongs to the set on the left then we must necessarily have the relation $$[\alpha(t_0)bb^\star (t_0)]^2=\alpha(t_0)(bb^\star (t_0))^2$$ 
	which forces $\alpha(t_0)=1$ hence proving our claim. We are now in a position to show that $a(b-a)=0$. Pick $x\in A$ arbitrary. If $t\in X\backslash K$ then we have that $(abx)(t)=(a^2x)(t)$; if $t\in K_a$ then $(abx)(t)=0=(a^2x)(t)$; if $t\in K_b$ then  $(abx)(t)=0$ but we have no information about $(a^2x)(t)$. However, this argument suffices to conclude that $\sigma(abx)\subseteq\sigma(a^2x)$ (for any $x\in A$). On the other hand \eqref{contain2} says $\sigma^\prime(a^2x)\subseteq\sigma^\prime(abx)$ for an arbitrary $x$. Since $A$ is commutative and $a\notin G(A)$ we actually have $\sigma(a^2x)\subseteq\sigma(abx)$ for all $x\in A$. By a result of Braatvedt and Brits \cite[Theorem 2.6]{univari} it follows that $(b-a)a=0$. 
\end{proof}

\begin{theorem}\label{Cstar}
	Let $A$ be any unital $C^\star$-algebra. If $a,b\in A$ satisfy $\sigma^\prime(ax)\subseteq\sigma^\prime(bx)\mbox{ for all }x\in A$, then $a$ is an algebraic truncation of $b$. More generally, \eqref{contain2} is equivalent to the condition that $ax$ is an algebraic truncation of $bx$ for all $x\in A$.
\end{theorem}
\begin{proof}
	We shall first establish the result under the assumption that $a$ and $b$ are normal: From Theorem~\ref{commute} we have that $a$ commutes with $b$ and $b^\star $. Observe further that \eqref{contain2} implies  $\sigma^\prime(a^\star x)\subseteq\sigma^\prime(b^\star x)\mbox{ for all }x\in A,$ and so, arguing as before, we have that $a^\star $ commutes with $b$ and $b^\star $. Let $B$ be the Banach algebra generated by the set $\{\mathbf 1, a, a^\star , b, b^\star \}$. Then
	$$\sigma_B(ax)=\sigma_A(ax)\subseteq\sigma_A(bx)=\sigma_B(bx)\mbox{ for each }x\in B.$$ Since $B$ is a commutative $C^\star $-algebra, it follows from  the Gelfand-Naimark Theorem together with Theorem~\ref{C(X)}, that $a(b-a)=(b-a)a=0$.\
	Suppose next that $a,b$ are arbitrary elements of $A$ satisfying \eqref{contain}. Then it follows that
	\begin{equation}\label{invol}
	\sigma^\prime(xa^\star )\subseteq\sigma^\prime(xb^\star )\ \forall\ x\in A.
	\end{equation}
	By Corollary~\eqref{cyc} $a^\star xb^\star =b^\star xa^\star $ for all $x\in A$. Observe then that 
	$$(ba^\star )(ba^\star )^\star =ba^\star ab^\star =ab^\star ba^\star =(ba^\star )^\star (ba^\star )$$ and hence $ba^\star $ is normal (if we then notice that \eqref{invol} implies that $\sigma(ba^\star )$ is on the real line we may deduce that $ba^\star $ is self-adjoint, and in fact  $ba^\star \geq0$). Then
	\begin{equation*}
	\sigma^\prime(xaa^\star )=\sigma^\prime(a^\star xa)\subseteq\sigma^\prime(b^\star xa)=\sigma^\prime(xab^\star )=
	\sigma^\prime(xba^\star )\subseteq\sigma^\prime(xbb^\star )
	\end{equation*}
	holds for all $x\in A$. Applying the result obtained in the first part of the proof to
	the self-adjoint elements $aa^\star $ and $ba^\star $ we conclude that $(aa^\star )(ba^\star )=(aa^\star )^2$.  Observe also that $(ba^\star )^2=ba^\star ab^\star =aa^\star bb^\star $. Now take the Banach algebra, say $B$, generated by the self-adjoint, mutually commuting, collection $\{\mathbf 1, ba^\star , aa^\star , bb^\star \}$. Again by the Gelfand-Naimark Theorem we may view $B$ as a $C(K)$ for some compact Hausdorff space $K$, whence $B$ is semisimple. If $\chi$ is any character of $B$ then, since
	$$(aa^\star )(ba^\star )=(aa^\star )^2\mbox{ and }(ba^\star )^2=(aa^\star )(bb^\star ),$$ it follows that
	$\chi(aa^\star )=\chi(ba^\star )$ and, by semisimplicity of $B$, we conclude that  $aa^\star =ba^\star $. To complete the proof: 
	\begin{align*}
	\|a(b-a)\|^2&=\|[a(b-a)][a(b-a)]^\star \|=
	\|a(b-a)(b-a)^\star a^\star \|\\&=
	\|(b-a)a(b-a)^\star a^\star \|=
	\|(b-a)(ab^\star -aa^\star )a^\star \|\\&=
	\|(b-a)(ba^\star -aa^\star )a^\star \|=0.
	\end{align*}
\end{proof}

\bibliographystyle{amsplain}
 \bibliography{Spectral}

\end{document}